\documentclass[10pt]{amsart}
\usepackage{amsmath}
\usepackage{amsfonts}
\usepackage{amssymb}
\usepackage{graphicx}
\usepackage{amsthm,graphicx,color,yfonts}
\usepackage{pdfsync}
\usepackage{epstopdf}%
\usepackage{pifont}
\usepackage{bbm}

\usepackage[colorlinks=true]{hyperref}
\hypersetup{linkcolor=red,citecolor=blue,filecolor=dullmagenta,urlcolor=darkblue} % coloured links
\newcommand{\R} {\mathbb R}
\newcommand{\cuad}{{\sqcap\kern-.68em\sqcup}}

\newcommand{\ve}{\varepsilon}

\newcommand{\be}{\begin{equation}}
\newcommand{\ee}{\end{equation}}

\newcommand{\la}{\lambda}

\definecolor{darkgreen}{rgb}{0.2,0.7,0.1}

\newcommand{\sech}{\mathop{\mbox{\normalfont sech}}\nolimits}

\newcommand{\al}{\alpha}

\def\bm{\left( \begin{array}{cc}}
\def\endm{\end{array}\right)}
\newcommand{\ba}{\begin{equation*}}
\newcommand{\ea}{\begin{equation*}}
\newcommand{\bea}{\begin{eqnarray}}
\newcommand{\eea}{\end{eqnarray}}
\newcommand{\bee}{\begin{eqnarray*}}
\newcommand{\eee}{\end{eqnarray*}}
\newcommand{\ben}{\begin{enumerate}}
\newcommand{\een}{\end{enumerate}}

\numberwithin{equation}{section}

\newtheorem{theorem}{Theorem}[section]
\newtheorem*{theorem*}{Theorem}

\newtheorem{corollary}{Corollary}[section]
\newtheorem{lemma}{Lemma}[section]

\theoremstyle{remark}
\newtheorem{remark}{Remark}[section]
\title[Decay for Boussinesq equations]{Scattering in the energy space for Boussinesq equations}

\author{Claudio Mu\~noz}
\address{CNRS and Departamento de Ingenier\'{\i}a Matem\'atica and Centro
de Modelamiento Matem\'atico (UMI 2807 CNRS), Universidad de Chile, Casilla
170 Correo 3, Santiago, Chile.}
\email{cmunoz@dim.uchile.cl}
\thanks{C. M. work was partly funded by Chilean research grants FONDECYT  1150202, Fondo Basal CMM-Chile, MathAmSud EEQUADD and Millennium
Nucleus Center for Analysis of PDE NC130017.}

\author{Felipe Poblete}
\address{Universidad Austral de  Chile, Facultad de Ciencias, Instituto de Ciencias F\'isicas y Matem\'aticas,  Valdivia, CHILE.}
\email {felipe.poblete@uach.cl}
\thanks{F. P. is partially supported by Chilean research grant FONDECYT  1170466 and DID S-2017-43 (UACh).}

\author{Juan C. Pozo}
\address{Departamento de Matem\'atica y Estad\'istica, Facultad de Ciencias, Universidad de La Frontera, Casilla 54-D, Temuco, Chile}
\email{juan.pozo@ufrontera.cl}
\thanks{J. C. Pozo is partially supported by  Chilean research grant FONDECYT  11160295.}


\begin{document}

%%%%%%%%%%%%%%%%%%%%%%%%%%%%%%%%%%%%%%%%%%%%%%%%%%%%%%%%%%%%%%%%%%%%%%%%%%%%%%%%%%%%%%%%%%%%%%%%%%
\begin{abstract}
In this note we show that all small solutions in the energy space of the generalized 1D Boussinesq equation must decay to zero as time tends to infinity, strongly on slightly proper subsets of the space-time light cone. Our result does not require any assumption on the power of the nonlinearity, working even for the supercritical range of scattering. For the proof, we use two new Virial identities in the spirit of works \cite{KMM,KMM1}. No parity assumption on the initial data is needed.
\end{abstract}

\maketitle

\section{Introduction and Main Results}

\medskip

In this paper we study a class of fourth order nonlinear wave equations appearing as a standard model in Physics.
More precisely, we consider the generalized (good) Boussinesq model \cite{Bous}
\begin{equation}\label{boussinesq}
\partial_t^2 u  + \partial_x^{4} u -  \partial_x^2 u + \partial_x^{2}f(u) =0, \quad (t,x)\in \R\times\R.
\end{equation}
Here $u=u(t,x)$ is a real-valued function. This equation appears as a canonical model of shallow water waves as well as the Korteweg-de Vries (KdV) equation, see e.g. \cite[p. 53]{Saut}. The fundamental works of Bona and Sachs \cite{Bona}, Linares  \cite{Linares}, and Liu \cite{Liu1,Liu2}, established that \eqref{boussinesq} is locally well-posed (and even globally well-posed for small data \cite{Bona,Linares}) in the standard energy space for $(u,\partial_t u) \in H^1\times L^2$. We assume that the smooth nonlinearity is of power type, in the  sense that for some $p>1$,
\be\label{Nonlinearity}
f(0)=0, \quad |f'(s)| \lesssim  |s|^{p-1},   \quad |s|<1.
\ee
The purpose of this paper is to show that the model \eqref{boussinesq} shares important similarities with KdV and (second order) Klein-Gordon equations in their long time decay and behavior. We will prove, using well-chosen Virial identities, that for \eqref{boussinesq}, small globally defined solutions (in the energy norm) must decay to zero locally in space. This being said, we prove these results independently of the subcritical, critical or supercritical character of the scattering mechanism for low powers of $p$ (a.k.a. the Strauss exponent).

\subsection{Main result} Before stating our result we need some standard notation. The Boussinesq model \eqref{boussinesq} can be written as follows: if $u_1 := u$, then
\be\label{Bous}
\begin{cases}
\partial_t u_1 = \partial_x u_2,\\
\partial_t u_2 = \partial_x (u_1- \partial_x^2 u_1-f(u_1)).
\end{cases}
\ee
\begin{theorem}\label{TH1}
There exists an $\ve>0$ such that if
\[
\|(u_1,u_2)(\cdot, 0)\|_{H^1\times L^2}< \ve,
\]
then one has, for any $C>0$ arbitrarily large and $I(t):= \Big(- \frac{Ct}{\log^2 t}, \frac{Ct}{\log^2 t}\Big)$,
\be\label{Conclusion_0}
\lim_{t \to\infty}   \|(u_1,u_2)(t)\|_{(H^1\times L^2)(I(t))} =0.
\ee
A similar result holds for the case $t\to-\infty$ after a suitable redefinition of $I(t)$.
\end{theorem}

\begin{remark}
By a result of Linares \cite{Linares} and Liu \cite{Liu1} (see also \cite[Theorems 3.1 and 3.2]{Liu2}), all small $H^1\times L^2$ solutions are globally defined, thanks to the conservation of the energy
\[
E[u,\partial_t u](t):= \frac12\int (u_2^2 + (\partial_x u_1)^2 + u_1^2)(t,x)dx - \int F(u_1)(t,x)dx,
\]
and the smallness assumption on the initial data. More precisely, we have the equivalence $E[u,\partial_t u](t) \sim \|(u_1,u_2)(t)\|_{H^1\times L^2}^2$ with implicit constants independent of time.
\end{remark}

Previous results on scattering of small amplitude solutions of \eqref{Bous} were proved by  Liu \cite{Liu1}, Linares-Sialom \cite{LS}, and Cho-Ozawa \cite{Cho}.
These contributions are mainly based either on the use of weighted Sobolev norms, or mixed $W^{s,p}$ spaces, and the additional condition $p \geq p_c$ (a critical power exponent) is needed to ensure either standard ($p>p_c$) or modified ($p=p_c$) scattering.

\medskip

Theorem \ref{TH1} shows full scattering to zero in the energy space and in any slightly proper subset of the light cone. It also improves previous decay estimates in \cite{Liu1,LS,Cho} in several directions. First of all, it does not require the assumption $p>p_c$ (the critical exponent for standard scattering results in the literature). Second, Theorem \ref{TH1} describes not only linear but also ``nonlinear scattering'' on compact sets of space, in the sense that small solitary waves (if any) do ``scatter'' to infinity following \eqref{Conclusion_0}, see \eqref{Soliton} for more details. Finally, Theorem \ref{TH1} only needs data in the energy space.

\medskip

Let us remark that Theorem \ref{TH1} gives no information on the remaining (unbounded) portion of the space, but since small solitary waves seem to persist in time \cite{Bona,Liu1}, it is unlikely to have linear scattering only as reminder term in \eqref{Conclusion_0} if one works in the energy space. However, a particular integral rate of decay can be obtained for the pair $(u_1,u_2)(t)$: for any $\la_0>0$ sufficiently large,
\[
\int_2^{\infty}  \int \sech^2\Big(\frac x{\la_0} \Big) ( ( \partial_x u_1)^2+u_1^2+u_2^2)(t,x)dx \, dt \lesssim  \la_0 \ve^2,
\]
(see \eqref{Integration_2}) as well as others mild decay estimates depending on time-depending weights. This ensures that both $u_1$ and $u_2$ are locally square integrable in time and space.

\medskip

Theorem \ref{TH1} is also in concordance with the existence of ``arbitrary size'' solitary waves for \eqref{Bous}. Indeed, assume that $f(u) = |u|^{p-1}u$, $p>1$ in \eqref{Bous}. Let $Q=Q(s)$ be the standard soliton given by
\[
Q(s) := \left( \frac{p+1}{2\cosh^2\left(\frac{(p-1)}{2}s\right)} \right)^{\frac1{p-1}}.
\]
Note that $Q$ solves $Q''-Q + Q^p =0$. Then, for any speed $|v|<1$ and $x_0\in \R$, the family\footnote{Note that the ``Lorentz boost'' is completely different to the one shared by second order scalar field equations; this is because \eqref{Bous} does not preserve the standard Lorentz invariance.}
\be\label{Soliton}
\begin{aligned}
Q_{v,x_0}(t,x) :=&~  \Big(\gamma^{\frac{2}{p-1}}Q(\gamma(x-vt-x_0)), -v\gamma^{\frac{2}{p-1}}Q(\gamma(x-vt-x_0)) \Big), \\
 \gamma:= &~ \sqrt{1-v^2},
 \end{aligned}
\ee
is a solitary wave for \eqref{Bous} \cite{Liu1}, \cite[p. 52]{Liu2}. Since small energy solitary waves must necessarily have ultra-relativistic speeds  ($|v| \sim 1$), Theorem \ref{TH1} must be valid only in the sub-relativistic regime. Note also that slow-speed solitary waves have sufficiently large energy to be ruled out by the hypothesis of Theorem \ref{TH1} (they also are unstable, \cite{Liu1}). For further information about the stability theory of \eqref{Soliton}, see \cite{Bona,Liu1}. See also \cite{AM,AMP,MR678151,MM1,MM2} for other similar stability results in other dispersive or scalar field equations.

\begin{remark}
In particular, Theorem \ref{TH1} precludes the existence of small $H^1\times L^2$ standing waves or ``breathers'' \cite{AM,AMP,KMM1} in \eqref{boussinesq} by purely dynamical methods. Even small nonlinear objects moving at speeds below $\sim t \log^{-2} t$ are ruled out by Theorem \ref{TH1}.
\end{remark}

The proof of Theorem \ref{TH1} follows the recent ideas introduced by Kowalczyk, Martel and the first author \cite{KMM,KMM1} for the case of second order scalar field equations, which are respectively based in fundamental works by Martel and Merle \cite{MM,MM1,MM2}, and Merle and Rapha\"el \cite{MR}. In both cases \cite{KMM,KMM1}, decay is showed using well-cooked Virial identities adapted to each model, and under the additional assumption of small odd data perturbations. Here in Theorem \ref{TH1} that condition is no longer needed because of some ``KdV dynamics''  hidden in the wave equation \eqref{Bous} which preserves a particular direction of movement in the dynamics (a ``decay of momentum''). In this work we introduce two different Virial identities, one for showing decay of $u_1(t)$, and a second one which shows a smoothing effect hidden in \eqref{Bous}, as well as decay for $u_2(t)$. For further scattering results around the zero state in scalar field equations, see \cite{Bam_Cucc,LS1,LS2,LS3,Lin_Tao,MR1681113,Ste} and references therein. This list is by no means exhaustive.

\begin{remark}
The proof of Theorem \ref{TH1} also reveals a hidden KdV character of \eqref{boussinesq}, probably well-known in the literature, but useful to understand why Theorem \ref{TH1} is valid for all kind of data in the energy space (unlike the results in \cite{KMM1}, which needed an oddness assumption). Formally, \eqref{boussinesq} can be written as
\[
\partial_x\Big( \partial_t( \partial_t \partial_x^{-1}u)  + \partial_x \Big(\partial_x^{2} u -   u + f(u)\Big) \Big) =0,
\]
so after dropping the $\partial_x$ operator in front, becomes a natural KdV-like equation, with the role of $u$ also played by $ \partial_t \partial_x^{-1}u$, just as in \eqref{Bous}.
\end{remark}

\begin{remark}
The interval $I(t)$ in Theorem \ref{TH1} can be slightly improved: $ I(t) = \Big(- \frac{Ct}{\log^{1+\ve} t}, \frac{Ct}{\log^{1+\ve} t}\Big)$, or  $ I(t) = \Big(- \frac{Ct}{\log t \log^{1+\ve}\log t}, \frac{Ct}{\log t\log^{1+\ve} \log t}\Big)$, $\ve>0$ are also completely valid regions for scattering. However, we cannot get the validity of Theorem \ref{TH1} inside the interval $ I(t) = \Big(- \frac{Ct}{\log t}, \frac{Ct}{\log t}\Big)$.
\end{remark}

\begin{remark}
We expect that some of the conclusions of Theorem \ref{TH1} could be available for the fourth order nonlinear wave model \cite{Bretherton,Pausader}
\be\label{NLW}
\partial_t^2 u + \partial_x^4 u + mu - f(u) =0, \quad m\in \R,
\ee
but with harder proofs, because of the lack of particular momentum decay, just as in \cite{KMM1}. Note that this last model and \eqref{boussinesq} are formally related by an homotopy through the \emph{fractional} Laplacian
\[
\partial_t^2 u + \partial_x^4 u + m (-\partial_x^2)^{\al} u - (-\partial_x^2)^{\al} f(u) =0, \quad \quad \al \in [0,1], \quad m=1.
\]
Also, it is well-known that \eqref{NLW} may have solitary wave solutions. Finally, see \cite{AM1} for a recent application of this technique to a quasilinear 1+1 model.
\end{remark}

\subsection{Organization of this paper} This paper is organized as follows: Section \ref{VIRIAL} deals with a Virial identity needed for the proof of Theorem \ref{TH1}. Then in Section \ref{3} we prove a first part of Theorem \ref{TH1}. Next, in Section \ref{4} we prove new Virial identities and a new smoothing estimate. Finally, Section \ref{5} is devoted to the last part of Theorem \ref{TH1}, involving the decay of $(u_1,u_2)$.

\bigskip

\section{A Virial identity}\label{VIRIAL}

We start with the following result (see \cite{KMM,KMM1} for more details).

\begin{lemma}\label{Virial_bous}
Let $(u_1,u_2)\in H^1\times L^2$ a solution of \eqref{Bous}. Consider $\psi =\psi(x)$ a smooth bounded function to be chosen later, and consider $\la(t)$ a never zero time scaling. Then for any $t\in \R$ we have
\be\label{Virial_Bous}
{\small\begin{aligned}
\frac{d}{dt} \int \psi\Big(\frac x{\la(t)} \Big) u_1u_2 = &~{}- \frac{\la'(t)}{\la(t)}\int \frac{x}{\la(t)}\psi'\Big(\frac x{\la(t)} \Big) u_1u_2 -\frac1{2\la(t)} \int \psi'\Big(\frac x{\la(t)} \Big) u_2^2\\
&~ {}-\frac1{2\la(t)} \int \psi'\Big(\frac x{\la(t)} \Big) u_1^2 + \frac1{2\la^3(t)}\int  \psi^{(3)}\Big(\frac x{\la(t)} \Big) u_1^2\\
&~{} -\frac3{2\la(t)} \int  \psi'\Big(\frac x{\la(t)} \Big) ( \partial_x u_1)^2\\
&~{}+\frac1{\la(t)} \int \psi'\Big(\frac x{\la(t)} \Big) (u_1 f(u_1) -F(u_1)).
\end{aligned}}
\ee
Here, $F(s)$ stands for $\int_0^s f(r)dr.$
\end{lemma}
\begin{proof}
We compute using \eqref{Bous} and integrating by parts:
\[
\begin{aligned}
\frac{d}{dt} \int \psi\Big(\frac x{\la(t)} \Big) u_1u_2  =&~ - \frac{\la'(t)}{\la(t)}\int \frac{x}{\la(t)}\psi'\Big(\frac x{\la(t)} \Big) u_1u_2 + \int \psi\Big(\frac x{\la(t)} \Big) u_2 \partial_x u_2 \\
&~+ \int \psi\Big(\frac x{\la(t)} \Big) u_1 \partial_t u_2\\
=&~- \frac{\la'(t)}{\la(t)}\int \frac{x}{\la(t)}\psi'\Big(\frac x{\la(t)} \Big) u_1u_2 -\frac1{2\la(t)} \int \psi'\Big(\frac x{\la(t)} \Big) u_2^2 \\
&~ + \int \psi\Big(\frac x{\la(t)} \Big) u_1 \partial_x (u_1- \partial_x^2 u_1-f(u_1)) \\
=&~  \frac{\la'(t)}{\la(t)}\int \frac{x}{\la(t)}\psi'\Big(\frac x{\la(t)} \Big) u_1u_2 -\frac1{2\la(t)} \int \psi'\Big(\frac x{\la(t)} \Big) u_2^2 \\
&~ -\frac1{2\la(t)} \int \psi'\Big(\frac x{\la(t)} \Big) u_1^2 + \int \partial_x\Big( \psi\Big(\frac x{\la(t)} \Big) u_1 \Big) \partial_x^2 u_1 \\
& ~ + \int \partial_x\Big( \psi\Big(\frac x{\la(t)} \Big) u_1 \Big) f(u_1).
\end{aligned}
\]
The two last terms above can be estimated as follows:
\[
\begin{aligned}
 \int \partial_x\Big( \psi\Big(\frac x{\la(t)} \Big) u_1 \Big) \partial_x^2 u_1 =&~  \frac1{\la(t)}\int  \psi'\Big(\frac x{\la(t)} \Big) u_1 \partial_x^2 u_1\\
 & ~{}-\frac1{2\la(t)} \int  \psi'\Big(\frac x{\la(t)} \Big) ( \partial_x u_1)^2\\
 =&~  \frac1{2\la^3(t)}\int  \psi^{(3)}\Big(\frac x{\la(t)} \Big) u_1^2 -\frac3{2\la(t)} \int  \psi'\Big(\frac x{\la(t)} \Big) ( \partial_x u_1)^2.
\end{aligned}
\]
A further integration by parts gives
\[
\begin{aligned}
\int \partial_x\Big( \psi\Big(\frac x{\la(t)} \Big) u_1 \Big) f(u_1) =&~\frac1{\la(t)} \int \psi'\Big(\frac x{\la(t)} \Big) u_1 f(u_1)+\int  \psi\Big(\frac x{\la(t)} \Big) \partial_x F(u_1) \\
=&~ \frac1{\la(t)} \int \psi'\Big(\frac x{\la(t)} \Big) (u_1 f(u_1) -F(u_1)).
\end{aligned}
\]
Collecting the last identities, we get \eqref{Virial_Bous}.
\end{proof}

\begin{remark}
Note that \eqref{Bous} enjoys an interesting Virial identity. Almost every quadratic term has the correct sign, and bad terms are small compared with good ones. This behavior can be also found in KdV like equations (see \cite{MM,MM1} for instance). The introduction of the $\la(t)$ is done in order to encompass almost all the light cone.
\end{remark}

\bigskip

\section{Start of proof of the Theorem \ref{TH1}}\label{3}

\medskip

We only assume the case $t\to +\infty$, the opposite case ($t\to -\infty$) being a direct consequence of a completely similar argument.

\subsection{Choice of $\la(t)$ and $\psi(x)$} Consider \eqref{Virial_Bous} and assume, without loss of generality, that $t\geq 2$. Given any constant $C>0$, define
\be\label{lambda}
\la(t) :=  \frac{Ct}{\log^2 t},
\ee
and
\be\label{psi}
\psi(x) := \tanh (x), \quad \psi'(x) = \sech^2(x).
\ee
Note that
\be\label{lap_la}
\frac{\la'(t)}{\la(t)} = \frac1t \Big(1-\frac{2}{\log t}\Big) .
\ee

\begin{lemma}
There exists an increasing sequence of time $t_n\uparrow \infty$ such that
\be\label{On_a_sequence}
\int \sech^2 \Big( \frac{x}{\la(t_n)}\Big) (u_1^2+(\partial_x u_1)^2 +u_2^2)(t_n,x)dx \longrightarrow 0 \hbox{ as $n\to +\infty$.}
\ee
Moreover, we have
\be\label{Integration}
\int_2^{\infty}\frac1{\la(t)} \int \sech^2\Big(\frac x{\la(t)} \Big) ( ( \partial_x u_1)^2+u_1^2+u_2^2)(t,x)dx \lesssim  \ve^2.
\ee
\end{lemma}

\begin{proof}
Consider \eqref{Virial_Bous} with the choice of $\la(t)$ and $\psi(x)$ given in \eqref{lambda} and \eqref{psi}. First we estimate the term
\[
- \frac{\la'(t)}{\la(t)}\int \frac{x}{\la(t)}\psi'\Big(\frac x{\la(t)} \Big) u_1u_2.
\]
We claim that for some fixed constant $\tilde C>0$,
\be\label{Bad_1}
\Big| \frac{\la'(t)}{\la(t)}\int \frac{x}{\la(t)}\psi'\Big(\frac x{\la(t)} \Big) u_1u_2 \Big|~ \leq \frac{1}{4\la(t)} \int \sech^2\Big(\frac x{\la(t)} \Big) u_1^2 + \frac{\tilde C\ve^2}{ t \log^2 t}.
\ee
Indeed, using \eqref{lap_la},
\[
\begin{aligned}
\Big| \frac{\la'(t)}{\la(t)}\int \frac{x}{\la(t)}\psi'\Big(\frac x{\la(t)} \Big) u_1u_2 \Big|~ \leq & ~{} \frac1t \int  \frac{|x|}{\la(t)} \sech^2\Big(\frac x{\la(t)} \Big) |u_1u_2| \\
\leq &~\frac{\log^2 t}{8Ct} \int \sech^2\Big(\frac x{\la(t)} \Big) u_1^2 \\
&~ \quad +\frac{2C}{t \log^2 t} \int  \frac{|x|^2}{\la^2(t)} \sech^2\Big(\frac x{\la(t)} \Big) u_2^2\\
\leq &~\frac{1}{8\la(t)} \int \sech^2\Big(\frac x{\la(t)} \Big) u_1^2 \\
&~ \quad  +\frac{2C}{t \log^2 t} (\sup_{s\in \R} s^2 \sech^2(s)) \int   u_2^2\\
\leq &~\frac{1}{8\la(t)} \int \sech^2\Big(\frac x{\la(t)} \Big) u_1^2 + \frac{\tilde C\ve^2}{ t \log^2 t}.
\end{aligned}
\]

\medskip

Now we consider the second bad term,
\[
 \frac1{\la^3(t)}\int  \psi^{(3)}\Big(\frac x{\la(t)} \Big) u_1^2.
\]
For this term clearly we have the estimate (it is enough to take $\la(t)$ larger than a fixed constant for all large time)
\be\label{Bad_2}
\Big| \frac1{2 \la^3(t)}\int  \psi^{(3)}\Big(\frac x{\la(t)} \Big) u_1^2 \Big| \lesssim \frac1{2\cdot8\la(t)}\int  \sech^2\Big(\frac x{\la(t)} \Big) u_1^2.
\ee
Finally, we consider the nonlinear term
\[
\frac1{\la(t)} \int \psi'\Big(\frac x{\la(t)} \Big) (u_1 f(u_1) -F(u_1)).
\]
Since by hypothesis \eqref{Nonlinearity}, $|u_1 f(u_1) -F(u_1)| \lesssim |u_1|^{p+1}$, we have
\[
\begin{aligned}
\Big| \frac1{\la(t)} \int \psi'\Big(\frac x{\la(t)} \Big) (u_1 f(u_1) -F(u_1)) \Big| \lesssim &~  \|u_1(t)\|_{L^\infty}^{p-1} \times \frac1{\la(t)} \int  \sech^2\Big(\frac x{\la(t)} \Big) u_1^2\\
\lesssim &~ \frac{\ve^{p-1}}{\la(t)} \int  \sech^2\Big(\frac x{\la(t)} \Big) u_1^2.
\end{aligned}
\]
By taking $\ve >0$ small enough, we have
\be\label{Bad_3}
\Big| \frac1{\la(t)} \int \psi'\Big(\frac x{\la(t)} \Big) (u_1 f(u_1) -F(u_1)) \Big| \leq   \frac1{8\la(t)} \int  \sech^2\Big(\frac x{\la(t)} \Big) u_1^2.
\ee
Collecting estimates \eqref{Bad_1}, \eqref{Bad_2} and \eqref{Bad_3}, and replacing in \eqref{Virial_Bous}, we obtain
\[
\begin{aligned}
\frac{d}{dt} \int \psi\Big(\frac x{\la(t)} \Big) u_1u_2 \leq &~{}  -\frac1{2\la(t)} \int \sech^2\Big(\frac x{\la(t)} \Big) u_2^2-\frac1{16\la(t)} \int \sech^2\Big(\frac x{\la(t)} \Big) u_1^2 \\
&~{} -\frac3{2\la(t)} \int  \sech^2\Big(\frac x{\la(t)} \Big) ( \partial_x u_1)^2 +\frac{\tilde C\ve^2}{ t \log^2 t} .
\end{aligned}
\]
Note that the two last terms above right are integrable in time. Consequently, we have \eqref{Integration}:
\[
\int_2^{\infty}\frac1{\la(t)} \int \sech^2\Big(\frac x{\la(t)} \Big) ( ( \partial_x u_1)^2+u_1^2+u_2^2)(t,x)dx \lesssim  \ve^2.
\]
Therefore, since $\la(t)^{-1}$ is not integrable in $[2,\infty)$, there exists a sequence of time $t_n\to +\infty$ (which can be chosen increasing after taking a subsequence), such that
%\[
%\lim_{n\to \infty}\frac1{\la(t_n)} \int \sech^2\Big(\frac x{\la(t_n)} \Big) ( ( \partial_x u_1)^2+u_1^2+u_2^2)(t_n,x)dx =0.
%\]
\[
\lim_{n\to \infty} \int \sech^2\Big(\frac x{\la(t_n)} \Big) ( ( \partial_x u_1)^2+u_1^2+u_2^2)(t_n,x)dx =0.
\]
The proof is complete.
\end{proof}

Let us make a small digression of the main proof. We recall that with a small modification of \eqref{Integration}, we can already show that $u_1(t)$ decays to zero in $H^1$ on compact sets of space (but we cannot show that $u_2(t)$ also decays to zero). Using similar arguments as in the previous proof (except that now \eqref{Bad_1} is not necessary), one can prove that

\begin{lemma}\label{Lema_A}
Let $\la_0>0$ be a large fixed constant. There exists an increasing sequence of time $t_n\uparrow \infty$ such that
\be\label{On_a_sequence_2}
\int \sech^2 \Big( \frac{x}{\la_0}\Big) (u_1^2+(\partial_x u_1)^2 +u_2^2)(t_n,x)dx \longrightarrow 0 \hbox{ as $n\to +\infty$.}
\ee
Moreover, one has the estimate equivalent to \eqref{Integration}
\be\label{Integration_2}
\int_2^{\infty} \int \sech^2\Big(\frac x{\la_0} \Big) ( ( \partial_x u_1)^2+u_1^2+u_2^2)(t,x)dx dt\lesssim \la_0  \ve^2.
\ee
\end{lemma}
These last estimates, although weaker than \eqref{On_a_sequence}, are enough to conclude the following result.

\begin{lemma}
For any compact interval $I\subset\R$, we have
\be\label{First_to_zero}
\lim_{t\to +\infty}\int_I u_1^2(t,x)dx = 0.
\ee
\end{lemma}
\begin{proof}
We have from \eqref{Bous},
\[
\frac{d}{dt} \Big(\int \sech^2\Big(\frac x{\la_0} \Big) u_1^2 \Big)=  \frac 2{\la_0}\int \sech^2\Big(\frac x{\la_0} \Big) u_1 \partial_x u_2,
\]
so that
\[
\Big| \frac{d}{dt}   \int \sech^2\Big(\frac x{\la_0} \Big) u_1^2 \Big| \lesssim  \frac1{\la_0}\int \sech^2\Big(\frac x{\la_0} \Big)( |\partial_x u_1|  + |u_1|)|u_2|.
\]
Integrating in time, we have
\[
\begin{aligned}
& \Big|  \int \sech^2\Big(\frac x{\la_0} \Big) u_1^2 (t_n) -   \int \sech^2\Big(\frac x{\la_0} \Big) u_1^2(t)  \Big|  \lesssim \\
& \qquad \lesssim  \int_{t}^{t_n}  \frac1{\la_0}\int \sech^2\Big(\frac x{\la_0} \Big)( (\partial_x u_1)^2 +u_1^2 + u_2^2)(s,x)dx ds.
\end{aligned}
\]
Sending $n\to \infty$, and using \eqref{Integration}, we get
\[
 \int \sech^2\Big(\frac x{\la_0} \Big) u_1^2(t) \lesssim  \int_{t}^{\infty}  \frac1{\la_0}\int \sech^2\Big(\frac x{\la_0} \Big)( (\partial_x u_1)^2 +u_1^2+ u_2^2)(s,x)dx ds.
\]
Finally, sending $t\to \infty$, we get
\[
\lim_{t\to +\infty} \int \sech^2\Big(\frac x{\la_0} \Big) u_1^2(t) =0,
\]
which implies \eqref{First_to_zero}.
\end{proof}

An easy consequence of this result is the following
\begin{corollary}
For each bounded interval $I$ one has $\|u_1(t)\|_{H^1(I)} \longrightarrow 0$ as $t\to \infty$.
\end{corollary}

\begin{proof}
Fix a bounded interval $I$. Take as sequence $t_n\to +\infty$, and consider the sequence $u_1(t_n)$, bounded in $H^1(\R)$. Take any subsequence (still denoted $u_1(t_n)$). Thanks to the compact embedding of $H^1(I)$ into $L^2(I)$ and \eqref{First_to_zero}, we have $u_1(t_n)$ convergent to zero in $H^1(I)$ (from the uniqueness of the limit). Since every subsequence has a subsequence convergent to the same limit, we conclude.
\end{proof}

In order to show in full generality the consequences of Theorem \ref{TH1}, we need additional estimates, part of the next Section.

\bigskip

\section{A second set of Virial identities}\label{4}

\medskip

In order to fully show Theorem \ref{TH1}, we need two additional Virial identities that will imply a new smoothing effect in \eqref{boussinesq}.  For $\la(t)>0$ as in  \eqref{lambda}, define
\be\label{Ip_Im}
\mathcal I_+(t) := \int \phi \partial_x u_1 u_2 dx, \quad \mathcal I_-(t):= - \int  \partial_x (\phi u_1) u_2 dx,
\ee
where, for the sake of simplicity, we have denoted
\be\label{phi}
\phi =\phi(t,x) := \frac1{\la(t)}\phi_0\Big(\frac x{\la(t)}\Big), \quad \phi_0:=\sech^2.
\ee
Note that both quantities $\mathcal I_+(t)$ and $\mathcal I_-(t)$ are well-defined for $H^1\times L^2$ solutions of \eqref{Bous}, and we have
\be\label{Bound0}
\sup_{t\in \R} ~(|\mathcal I_+(t)| +|\mathcal I_-(t)| ) \lesssim \ve^2.
\ee

\begin{lemma}[Second Virial identities]\label{Second}
Assume that $(u_1,u_2)(t)$ is a sufficiently smooth and decaying solution of \eqref{Bous}. Then we have
\be\label{I_p}
\begin{aligned}
\frac{d}{dt}\mathcal I_+(t) =  &~  \int \partial_t \phi \partial_x u_1 u_2  - \int \phi (\partial_x u_2)^2+\int \phi (\partial_x^2 u_1)^2  +\frac12 \int \partial_x^2 \phi u_2^2 \\
&~  +\int \Big(\phi -\frac12 \partial_x^2 \phi \Big)(\partial_x u_1)^2  -\int \phi (\partial_x u_1)^2 f'(u_1),
\end{aligned}
\ee
and
\be\label{I_m}
\begin{aligned}
\frac{d}{dt}\mathcal I_-(t) =  &~ - \int  \partial_x (\partial_t \phi u_1) u_2 +  \int \phi   (\partial_x u_2)^2+\int \phi (\partial_x^2 u_1)^2   \\
&~  -\int (\phi +2 \partial_x^2 \phi) (\partial_x u_1)^2   +\frac12 \int (\partial_x^2 \phi + \partial_x^4 \phi ) u_1^2\\
&~  +\int \partial_x \phi u_1 \partial_x u_1 f'(u_1) +\int \phi f'(u_1) (\partial_x u_1)^2.
\end{aligned}
\ee
\end{lemma}

\begin{proof}
First we prove \eqref{I_p}. We have
\[
\begin{aligned}
\frac{d}{dt}\mathcal I_+(t) = &~  \int \partial_t \phi \partial_x u_1 u_2  + \int \phi \partial_{tx} u_1 u_2 +  \int \phi \partial_{x} u_1 \partial_t u_2 \\
=  &~  \int \partial_t \phi \partial_x u_1 u_2  + \int \phi \partial_x^2 u_2  u_2 \\
&~ +\int \phi \partial_x u_1 \partial_x ( u_1 -\partial_x^2 u_1 -f(u_1))\\
= & ~ \int \partial_t \phi \partial_x u_1 u_2   -\int \phi ( \partial_x u_2)^2 +\frac12 \int \partial_x^2 \phi u_2^2 +\int \phi ( \partial_x u_1)^2 \\
&~  -\frac12 \int \partial_x^2 \phi ( \partial_x u_1)^2 +\int \phi ( \partial_x^2 u_1)^2 -\int \phi ( \partial_x u_1)^2 f'(u_1).
\end{aligned}
\]
Rearranging terms, we get \eqref{I_p}. Now, for the proof of \eqref{I_m}, we have
\[
\begin{aligned}
\frac{d}{dt}\mathcal I_-(t) = &~   \int   \partial_t \phi u_1\partial_x u_2 + \int \phi \partial_t (u_1 \partial_x u_2) \\
=& ~ - \int  \partial_x (\partial_t \phi u_1) u_2 +  \int \phi \partial_{t} u_1\partial_x  u_2 +  \int \phi u_1 \partial_{tx} u_2 \\
=  &~- \int  \partial_x (\partial_t \phi u_1) u_2 + \int \phi (\partial_x  u_2)^2   \\
&~ - \int  \partial_x(\phi u_1) \partial_x ( u_1 -\partial_x^2 u_1 -f(u_1))\\
= & ~ - \int  \partial_x (\partial_t \phi u_1) u_2 -\int \phi ( \partial_x u_2)^2 \\
&~  - \int  ( \partial_x\phi u_1 + \phi \partial_x u_1 )  ( \partial_xu_1 -\partial_x^3 u_1 -f'(u_1)\partial_x u_1).
\end{aligned}
\]
Consequently,
\be\label{auxiliar}
\begin{aligned}
\frac{d}{dt}\mathcal I_-(t)  =& ~ - \int  \partial_x (\partial_t \phi u_1) u_2  + \int \phi ( \partial_x u_2)^2  + \frac12 \int \partial_x^2 \phi u_1^2  -\int  \phi (\partial_x u_1 )^2 \\
& ~ +  \int  \partial_x ( \partial_x \phi u_1 + \phi \partial_x u_1 ) \partial_x^2 u_1 \\
& ~ + \int   \partial_x \phi u_1 \partial_x u_1f'(u_1) + \int   \phi f'(u_1)(\partial_x u_1)^2.
\end{aligned}
\ee
Finally, the term $\int  \partial_x ( \partial_x\phi u_1 + \phi \partial_x u_1 ) \partial_x^2 u_1$ can be reduced to
\[
\int  \partial_x ( \partial_x \phi u_1 + \phi \partial_x u_1 ) \partial_x^2 u_1 = - 2 \int \partial_x^2 \phi (\partial_x u_1)^2 +\frac12 \int \partial_x^4 \phi  u_1^2 +\int \phi (\partial_{x}^2 u_1)^2.
\]
Plugging this identity in \eqref{auxiliar}, and rearranging terms, we finally obtain \eqref{I_m}.
\end{proof}

Lemma \ref{Second} will be useful to prove a Kato-type local smoothing effect for $H^1\times L^2$ solutions of \eqref{Bous} (note that all computations are easily justified by a standard limiting argument).

\begin{corollary}\label{Integration_3a}
The following smoothing estimate holds
\be\label{Integration_3}
\int_2^{\infty} \frac1{\la(t)} \int \sech^2 \Big(\frac x{\la(t)} \Big) ( ( \partial_x^2 u_1)^2+(\partial_x u_2)^2)(t,x)dx dt<+\infty.
\ee
In particular, there exists an increasing sequence of time $s_n\uparrow \infty$ such that
\be\label{On_a_sequence_3}
\int \sech^2 \Big( \frac{x}{\la(s_n)}\Big) ( ( \partial_x^2 u_1)^2+(\partial_x u_2)^2)(s_n,x)dx \longrightarrow 0 \hbox{ as $n\to +\infty$.}
\ee
\end{corollary}

\begin{proof} In \eqref{I_p}, the only complicated term is $  \int \partial_t \phi \partial_x u_1 u_2$. For this term, we have
\[
\int \partial_t \phi \partial_x u_1 u_2 =  -\frac{\la'(t)}{\la(t)}\int \phi \partial_x u_1 u_2  -\frac{\la'(t)}{\la^2(t)}\int  \frac{x}{\la(t)} \phi_0' \Big(\frac x{\la(t)}\Big)\partial_x u_1 u_2.
\]
Using \eqref{lambda} and \eqref{lap_la},
\be\label{Aux_66}
\Big| \int \partial_t \phi \partial_x u_1 u_2 \Big|~ \lesssim  ~{}  \int \phi( (\partial_x u_1)^2 +u_2^2) + \frac{\ve^2}{t^2} \log^2 t.
\ee
On the other hand, in \eqref{I_m} the only complicated term is $- \int  \partial_x (\partial_t \phi u_1) u_2 $. Here we have
\[
\begin{aligned}
- \int  \partial_x (\partial_t \phi u_1) u_2  = &~  -\frac{\la'(t)}{\la(t)}\int \partial_x( \phi  u_1) u_2 \\
&~  -\frac{\la'(t)}{\la^2(t)}\int \partial_x \Big( \frac{x}{\la(t)} \phi_0' \Big(\frac x{\la(t)}\Big) u_1 \Big) u_2.
\end{aligned}
\]
Then, exactly as in the estimate \eqref{Aux_66}, we have
\be\label{Aux_77}
\Big|\int  \partial_x (\partial_t \phi u_1) u_2 \Big|~ \lesssim  ~{}  \int \phi( (\partial_x u_1)^2 +u_2^2) + \frac{\ve^2}{t^2} \log^2 t.
\ee
Hence, (using \eqref{Aux_66}-\eqref{Aux_77}) from the addition of \eqref{I_p} and \eqref{I_m},
\[
\begin{aligned}
\Big|\frac{d}{dt}\mathcal I_+(t) +\frac{d}{dt}\mathcal I_-(t) -2\int \phi (\partial_x^2 u_1)^2 \Big| \lesssim \int \phi ((\partial_x u_1)^2 +u_1^2 +u_2^2)+ \frac{\ve^2}{t^2} \log^2 t.
\end{aligned}
\]
and from the subtraction of \eqref{I_p} and \eqref{I_m},
\[
\begin{aligned}
\Big|\frac{d}{dt}\mathcal I_+(t) - \frac{d}{dt}\mathcal I_-(t) + 2\int \phi (\partial_x u_2)^2 \Big| \lesssim \int \phi ((\partial_x u_1)^2 +u_1^2 +u_2^2) + \frac{\ve^2}{t^2} \log^2 t.
\end{aligned}
\]
Therefore, using \eqref{Integration} and \eqref{Bound0}, we have
\[
\int_2^\infty \int \phi ( (\partial_x u_2)^2+ (\partial_x^2 u_1)^2) <\infty.
\]
Finally, \eqref{On_a_sequence_3} follows by a standard argument (see \eqref{On_a_sequence}).
\end{proof}

\bigskip
\section{End of proof of Theorem \ref{TH1}}\label{5}
\medskip
Now we end the proof of Theorem \ref{TH1}. Let
\be\label{phi_1}
\phi_1:= \sech^4 =\phi_0^2.
\ee
The power 4 is necessary because of a slight loss of decay in an estimate below. We will use a third energy estimate:

\begin{lemma}
Let $\phi_1$ be as in \eqref{phi_1} and $F$ such that $F'=f$ and $F(0)=0$. Then,
\be\label{final_1}
\begin{aligned}
&~ \frac d{dt} \frac12 \int \phi_1 \Big(\frac x{\la(t)}\Big)((\partial_x u_1)^2 +u_1^2 + u_2^2 - 2 F(u_1)) =\\
& \qquad = ~ \frac12 \int \partial_t \Big(\phi_1\Big(\frac x{\la(t)}\Big) \Big)((\partial_x u_1)^2 +u_1^2 + u_2^2 - 2 F(u_1)) \\
& \qquad \quad +  \frac{2}{\la(t)}\int \phi'_1\Big(\frac x{\la(t)}\Big)  u_2\partial_{x}^2 u_1 + \frac1{\la^2(t)} \int \phi_1'' \Big(\frac x{\la(t)}\Big)\partial_x u_1 u_2 \\
& \qquad \quad - \frac1{\la(t)}\int \phi_1' \Big(\frac x{\la(t)}\Big)u_1 u_2 + \frac1{\la(t)}\int \phi_1' \Big(\frac x{\la(t)}\Big)u_2 f(u_1).
\end{aligned}
\ee
\end{lemma}
\begin{proof}
We denote $\phi_1 =\phi_1 (\frac x{\la(t)})$ for simplicity, and we compute:
\[
\begin{aligned}
& \frac d{dt} \frac12 \int \phi_1 ((\partial_x u_1)^2 +u_1^2 + u_2^2 - 2 F(u_1))  =  \\
&~ =   \frac12 \int \partial_t \phi_1 ((\partial_x u_1)^2 +u_1^2 + u_2^2 - 2 F(u_1))\\
&~ \quad + \int \phi_1 (\partial_{xt}^2 u_1\partial_x u_1 +u_1 \partial_t u_1 + u_2 \partial_t u_2 -  f(u_1) \partial_t u_1).
\end{aligned}
\]
Integrating by parts,
\[
\begin{aligned}
& \frac d{dt} \frac12 \int \phi_1 ((\partial_x u_1)^2 +u_1^2 + u_2^2 - 2 F(u_1))  =  \\
&~ =  \frac12 \int \partial_t \phi_1 ((\partial_x u_1)^2 +u_1^2 + u_2^2 - 2 F(u_1))\\
&~ \quad +  \int \phi_1 (u_1 -\partial_x^2 u_1  -  f(u_1)) \partial_t u_1 - \int \partial_x \phi_1 \partial_x u_1 \partial_t u_1   + \int \phi_1 u_2 \partial_t u_2 \\
&~ =   \frac12 \int \partial_t \phi_1 ((\partial_x u_1)^2 +u_1^2 + u_2^2 - 2 F(u_1))\\
&~ \quad + \int \phi_1 (u_1 -\partial_x^2 u_1  -  f(u_1)) \partial_x u_2 - \int \partial_x\phi_1 \partial_x u_1 \partial_x u_2   + \int \phi_1 u_2 \partial_t u_2 .
\end{aligned}
\]
Using \eqref{Bous}:
\[
\begin{aligned}
& \frac d{dt} \frac12 \int \phi_1 ((\partial_x u_1)^2 +u_1^2 + u_2^2 - 2 F(u_1))  =  \\
& ~ = \frac12 \int \partial_t \phi_1 ((\partial_x u_1)^2 +u_1^2 + u_2^2 - 2 F(u_1))\\
&~\quad  - \int \partial_x \phi_1 (u_1 -\partial_x^2 u_1  -  f(u_1))  u_2 -  \int \phi_1 \partial_x (u_1 -\partial_x^2 u_1  -  f(u_1)) u_2 \\
&~  \quad   + \int \phi_1 u_2 \partial_t u_2- \int \partial_x\phi_1 \partial_x u_1 \partial_x u_2 \\
& ~ =  \frac12 \int \partial_t \phi_1 ((\partial_x u_1)^2 +u_1^2 + u_2^2 - 2 F(u_1))\\
&~\quad - \int \partial_x\phi_1 (u_1 -\partial_x^2 u_1  -  f(u_1))  u_2 - \int \partial_x\phi_1 \partial_x u_1 \partial_x u_2.
\end{aligned}
\]
Now we integrate by parts to obtain
\[
\begin{aligned}
&~  \frac d{dt} \frac12 \int \phi_1 ((\partial_x u_1)^2 +u_1^2 + u_2^2 - 2 F(u_1))  =\\
&~  \quad = \frac12 \int \partial_t \phi_1 ((\partial_x u_1)^2 +u_1^2 + u_2^2 - 2 F(u_1))\\
&~ \quad \quad - \int \partial_x\phi_1 u_1 u_2 + 2\int \partial_x\phi_1 u_2  \partial_x^2 u_1  + \int \partial_x\phi_1  f(u_1)  u_2 + \int \partial_x^2\phi_1 \partial_x u_1 u_2.
\end{aligned}
\]
Noticing that $\phi_1 =\phi_1 (\frac x{\la(t)})$, we get the result, as desired.
\end{proof}

Now we conclude the proof of Theorem \ref{TH1}.  First we have
\[
\begin{aligned}
& \int \partial_t \phi_1 ((\partial_x u_1)^2 +u_1^2 + u_2^2 - 2 F(u_1)) = \\
&~ =  -\frac{\la'(t)}{\la(t)}\int  \frac{x}{\la(t)} \phi_1' \Big(\frac x{\la(t)}\Big)((\partial_x u_1)^2 +u_1^2 + u_2^2 - 2 F(u_1)).
\end{aligned}
\]
Consequently, using that $\Big| \frac{\la'(t)}{\la(t)}  \frac{x}{\la(t)} \phi_1' (\frac x{\la(t)}) \Big|~ \lesssim ~\phi(t,x)$,
\[
\Big|\frac12 \int \partial_t \phi_1 ((\partial_x u_1)^2 +u_1^2 + u_2^2 - 2 F(u_1))\Big| \lesssim  \int \phi ( (\partial_x u_1)^2 +u_1^2 +u_2^2).
\]
Using this last estimate, we have from \eqref{final_1} and the crude estimate $|\partial_x\phi_1| \lesssim \phi$,
\[
\Big| \frac d{dt} \frac12 \int \phi_1 ((\partial_x u_1)^2 +u_1^2 + u_2^2 - 2 F(u_1)) \Big| \lesssim  \int \phi (  (\partial_x^2 u_1)^2+ (\partial_x u_1)^2 +u_1^2 +u_2^2).
\]
From Corollary \ref{Integration_3a} and \eqref{Integration} we get for $t<t_n$,
\[
\begin{aligned}
& \Big|  \int \phi_1 ((\partial_x u_1)^2 +u_1^2 + u_2^2 - 2 F(u_1))(t_n) -  \int \phi_1 ((\partial_x u_1)^2 +u_1^2 + u_2^2 - 2 F(u_1))(t)  \Big| \\
&~  \qquad \lesssim \int_t^{\infty} \!\! \int \phi (  (\partial_x^2 u_1)^2+ (\partial_x u_1)^2 +u_1^2 +u_2^2)<\infty
\end{aligned}
\]
Sending $t_n \to +\infty$ we have
\[
\int \phi_1 ((\partial_x u_1)^2 +u_1^2 + u_2^2 - 2 F(u_1))(t) \lesssim \int_t^{\infty} \!\! \int \phi (  (\partial_x^2 u_1)^2+ (\partial_x u_1)^2 +u_1^2 +u_2^2).
\]
Therefore
\[
\lim_{t\to +\infty} \int \phi_1 ((\partial_x u_1)^2 +u_1^2 + u_2^2 - 2 F(u_1)) (t,x)dx  =0.
\]
In particular, from the smallness assumption on the data and the Sobolev inequality,
\[
\lim_{t\to +\infty} \int \phi_1 ((\partial_x u_1)^2 +u_1^2 + u_2^2) (t,x)dx  =0.
\]
The conclusion in Theorem \ref{TH1} follows from the fact that $\la(t)$ given in \eqref{lambda} is such that
\[
\int \phi_1 ((\partial_x u_1)^2 +u_1^2 + u_2^2) (t,x)dx  \gtrsim  \int_{I(t)} ((\partial_x u_1)^2 +u_1^2 + u_2^2) (t,x)dx,
\]
with involved constant independent of time. The proof is complete.
\bigskip

\providecommand{\bysame}{\leavevmode\hbox to3em{\hrulefill}\thinspace}
\providecommand{\MR}{\relax\ifhmode\unskip\space\fi MR }
% \MRhref is called by the amsart/book/proc definition of \MR.
\providecommand{\MRhref}[2]{%
  \href{http://www.ams.org/mathscinet-getitem?mr=#1}{#2}
}
\providecommand{\href}[2]{#2}

\end{document}